\documentclass[11pt]{amsart}

\usepackage{amscd}
\usepackage{amsmath}
\usepackage{amsthm}
\usepackage{amsfonts}
\usepackage{eucal}
\usepackage[all]{xy}
\usepackage{xcolor}
\usepackage{soul}

\newtheorem{theorem}{Theorem}[section]
\newtheorem{lemma}[theorem]{Lemma}

\newtheorem{corollary}[theorem]{Corollary}

\theoremstyle{definition}

\newtheorem{definition}[theorem]{Definition}
\newtheorem{example}[theorem]{Example}

\theoremstyle{remark}

\newtheorem*{ack}{\bf Acknowledgement}

\newcommand{\field}[1]{\ensuremath{\mathbb{#1}}}
\newcommand{\A}{\field{A}}
\newcommand{\C}{\field{C}}

\newcommand{\N}{\field{N}}
\newcommand{\Pp}{\field{P}}

\newcommand{\Z}{\field{Z}}

\newcommand{\cala}{\mathcal{A}}

\newcommand{\calc}{\mathcal{C}}

\newcommand{\cale}{\mathcal{E}}
\newcommand{\call}{\mathcal{L}}

\newcommand{\calo}{\mathcal{O}}

\newcommand{\bone}{\mathbf{1}}

\def \P {\mathbb{P}}

\def \z {\mathbb{Z}}
\def \c {\mathbb{C}}

\def \a {\alpha}
\def \g {\gamma}

\def \s {\sigma}
\def \p {\pi}

\def \uno {1\!\!1}
\def \b {\bf b}
\def \e {{\bf e}}

\newcommand{\lp}{\left( }
\newcommand{\rp}{\right) }

\newcommand{\cvp}[2]{\ensuremath{{\mathcal C}_{#1}\bigl(#2\bigr)}}
\newcommand{\cvpd}[3]{\ensuremath{{\mathcal C}_{#1,#2}\left( #3\right)}}

\newcommand{\ax}[2]{\ensuremath{\Pi_{#1}\bigl(#2\bigr)}}
\newcommand{\alg}[2]{\ensuremath{\mathcal{A}_{#1}\bigl(#2\bigr)}}
\newcommand{\algp}[2]{\ensuremath{\mathcal{A}^{\geq}_{#1}\bigl(#2\bigr)}}

\newcommand{\ecs}[2]{\ensuremath{E_{#1}\bigl(#2\bigr)}}
\newcommand{\pb}[1]{\ensuremath{{\Pp}\bigl(#1\bigr)}}

\newcommand{\eone}{\ensuremath{E\oplus{\bone}}}

\newcommand{\aeq}{\sim_{\text{alg}}}

\DeclareMathOperator{\Pic}{Pic}

\title[Euler-Chow series for Scrolls and Ruled Surfaces]{Euler-Chow series for Scrolls and Ruled Surfaces}
\author{E. Javier Elizondo}
\address{Instituto de Matem\'aticas, UNAM, Mexico}
\email{javier@im.unam.mx}
\thanks{The first author was supported in part by grant UNAM-DGAPA IN102418}
\author{Eladio Escobedo }
\address{Instituto de Mate\'aticas, UNAM, Mexico}
\email{eladio@matem.unam.mx}
\thanks{The second author was partially supported by UNAM-DGAPA IN102418}

\begin{document}

\begin{abstract}
We prove that the Euler-Chow series for ruled surfaces and scrolls is rational by means of an explicit computation. 
\end{abstract}

\maketitle
\tableofcontents

\section{Introduction}
\label{sec:intro}

Chow varieties have been of great interest to mathematicians for a long time. However, we do not know much about them. The work of H. Blaine Lawson was a breakthrough in the subject and it had many implications in the area. But many other aspects of these varieties are still not known. 

In this paper, we are interested in computing what has been called the Euler-Chow series. Originally the series came up as a way to compute the Euler characteristic of Chow varieties, but later on it was found that the series has a very interesting properties on its own. The first example of a computation of the series was done by Blaine Lawson and S.S.T. Yau (see \cite{law&yau-hosy}). In this case the series was a formal power series whose coefficients are the Euler characteristic of Chow varieties in the projective space. 

Following similar ideas, the series was formalized and computed for all dimensions for simplicial projective toric varieties by E. Javier Elizondo (see \cite{eli-tor}). The rationality of the series was proved and computed directly from the fan associated to the toric variety. 

In \cite{eli-lim} E. Javier Elizondo and Paulo Lima-Filho observe that if the Picard group of a projective variety is isomorphic to the integer numbers then the Euler-Chow series is the Hilbert series. Additional insights in this paper show that the Euler-Chow series generalizes other series of interest in algebraic geometry.  

In \cite{eli-kim}E. Javier Elizondo an Sun-ichi Kimura introduced the motivic version of the Euler-Chow series, where instead of taking the Euler-characteristic of Chow varieties as the coefficients of the series, they took the pure motive of Chow varieties. In dimension zero, the series generalizes the Weil Zeta series and Kapranov's. But it has the advantage that it is defined for all dimensions. Some computation were made there, the series was computed for blow ups on a finite number of points in a line in the projective plane. In this case the motivic version agrees with the topological one (Euler characteristic of Chow varieties). 

 There is a strong relation between the Cox ring and the Euler-Chow series. The series is irrational for a family of varieties where the Cox ring is not noetherian.  In \cite{xi-javier} X. Chen and E. J. Elizondo establish a conjecture for the rationality of the series with respect the Cox ring. 
 
 In this paper we compute the Euler-Chow series for ruled surfaces and normal rational scrolls. Since we give an explicit generating function for the series, we are computing the Euler characteristic of the Chow varieties of ruled surfaces and normal rational scrolls. This is the first time that the series is computed in codimension one where the Picard group is not discrete. However, the computation shows that the series behave quite well, and the formulas that are obtained are a natural generalization of previous cases, for example, the case of Hirzebruch surfaces (toric varieties). 

Following  \cite{eli-lim} recall that 
if $X$ is a projective variety and  $\calc_p(X)$ denotes the monoid of effective $p$-cycles on $X$,
writing $\Pi_p(X) := \pi_0(\calc_p(X))$ the path-connected components of $\calc_p(X)$, the $p$-Euler-Chow series is defined by 
\begin{equation}
\label{def:ECo}
E_p (X) = \sum_{\lambda \in \Pi_{p}(X)} \chi (C_{\lambda}(X))\cdot t^{\lambda}
\end{equation}
Here, $\chi (C_{\lambda}(X))$ is the Euler characteristic  of the space $ C_{\lambda}(X)$ of effective cycles with class $\lambda$. We can consider $E_p(X)$ as a function from $\Pi_p(X)$ to $\Z$, which sends $\lambda$ to $\chi(C_{\lambda})$. Thus, the series $ E_p (X)$ lives in the monoid ring $\Z [\Pi_p(X)] := \Z^{\Pi_p(X)}$ endowed with the convolution product (product of series), which is well defined because $\Pi_p (X)$ has  finite fibers with respect the sum. 

The easiest examples that we can find are the followings
\newpage

\begin{example}\hfill

\noindent{\bf 1.}\ If $X$ is a connected variety, then $\cvp{0}{X} =
\coprod_{d\in \Z_+} SP_d(X)$, where $SP_d(X)$ is the $d$-fold symmetric
product of $X$. Therefore, the $0$-th Euler-Chow series is given by
$$\ecs{0}{X} = \sum_{d\geq 0} \chi( SP_d(X) ) \ t^d = \lp \frac{1}{1-t}
\rp^{\chi(X)},$$
according to McDonald's formula \cite{mcd-sym}.

\noindent{\bf 2.}\  For $X = \Pp^n$, one has $\ax{p}{\Pp^n} \cong \Z_+$,
with the isomorphism given by the degree of the cycles. In this case, the
$p$-th Euler-Chow series was computed in \cite{law&yau-hosy}:
$$\ecs{p}{\Pp^n} = \sum_{d\geq 0} \chi(\cvpd{p}{d}{\Pp^n} ) \ t^d =
\lp \frac{1}{1-t} \rp^{\binom{n+1}{p+1}}.$$

\noindent{\bf 3.}\ Suppose that $X$ is an $n$-dimensional variety such that
$\Pic(X) \cong \Z$, generated by a very ample line bundle $L$. Then, 
$\ax{n-1}{X} \cong \Z_+$ and 
$\ecs{n-1}{X}$ is precisely the Hilbert series
for the projective embedding  of $X$ induced by $L$.
\end{example}

The paper is organized as follows:
\begin{enumerate} 
\item In Section \ref{projective-formulas} we prove a formula that express the Euler-Chow series of the projectivization of a finite sum of vector bundles over a variety in terms of the Euler-Chow series of products of these vector bundles. 
\item Section \ref{ruled-surfaces} is dedicated to compute the series for ruled surfaces. We obtain an explicit formula in Theorem \ref{thm:ruled}. 
\item In Section \ref{scrolls} we compute the Euler-Chow series for normal rational scrolls. Our main result is Corollary \ref{scrolls:cor} in page \pageref{scrolls:cor}. At the end we compute an example.
\end{enumerate}

\begin{ack}
We like to thank Felipe Zaldivar who kindly read the article and whose revisions and comments have been of great help. 
\end{ack}

\section{Projective bundle formulas}
\label{projective-formulas}

Before we state the main theorem in this section, we introduce some of the vocabulary in order to understand the notation in the theorem. 

Given monoids $M$ and $N$, and a commutative ring $S$, one can define a map $ \odot: S^M \otimes_S S^N  \; \rightarrow \; S^{M\times N }$. This map sends $f\otimes g$ to the function $f\odot g \; \in \; S^{M\times N }$\ which assigns to $(m,n)$ the element $f(m)g(n) \in S.$
Also, a monoid morphism $\Psi : M \longrightarrow N$,
induces a map from $S^M$ to $S^N$ sending $f\in S^M$ to $\Psi_{\sharp}f \in S^N$
 giving by
$$
\lp  \Psi_{\sharp}f\rp  (n) \,\, =
 \sum_{m\in\Psi^{-1}(n)} \, f(m)
$$
if $\Psi$ has finite fibers. This is a homomorphism of rings.

The main theorem in this section is 
\begin{theorem}
\label{product:thm}
Let  $E_{1}, E_{2} \mbox{ and } E_{3}$ be vector bundles over a complex projective variety 
 $W$, of ranks $e_{1}, e_{2} \mbox{ and } e_{3}$, respectively, and let $2\leq p\leq e_{1}+e_{2}+e_{3}-1$. 
 Then the $p$-th Euler-Chow series of  $\P(E_{1}\oplus E_{2}\oplus E_{3})$ is given by
$$E_{p}(\P(E_{1}\oplus E_{2}\oplus E_{3}))=\Psi_{p\#}\left(
\begin{array}{c}
E_{p-2}(\P(E_{1})\times_{W}\P(E_{2})\times_{W}\P(E_{3}))\odot E_{p-1}(\P(E_{1})\times_{W}\P(E_{2})) \\
\odot E_{p-1}(\P(E_{1})\times_{W}\P(E_{3}))\odot E_{p-1}(\P(E_{2})\times_{W}\P(E_{3})) \\
\odot E_{p}(\P(E_{1}))\odot E_{p}(\P(E_{2}))\odot E_{p}(\P(E_{3}))
\end{array}
\right).$$
\end{theorem}
 This theorem is a generalization of Theorem 5.1 in (\cite{eli-lim}) which is the formula when there are only two vector bundles. Indeed, if $E_3 \cong 0$, and we denote $\Pp^{-1} := \Pp(E_3)= \emptyset$  the fibre product with $\Pp(E_3) = \emptyset$ and $E_p(\emptyset) =1$ we obtain the formula for two vector bundles (Theorem 5.1).

Before going through the proof, we start by defining the morphism $\Psi_p$ that appears in the theorem. 

Here,  $E_{1}, E_{2} \mbox{ and } E_{3}$ are vector bundles over a complex projective variety $W$. We have the following maps: 
$$\begin{array}{rrclc}
i_{k}:&\P(E_{k}) & \rightarrow & \P(E_{1}\oplus E_{2}\oplus E_{3}) & (k=1,2,3) \\
i_{kl}:&\P(E_{k}\oplus E_{l}) & \rightarrow & \P(E_{1}\oplus E_{2}\oplus E_{3}) & (k<l,\ k,l=1,2,3) \\
p_{k}:&\P(E_{k}) & \rightarrow & W & (k=1,2,3) \\
p_{kl}:&\P(E_{k}\oplus E_{l}) & \rightarrow & W & (k<l,\ k,l=1,2,3) \\
q:&\P(E_{1}\oplus E_{2}\oplus E_{3}) & \rightarrow & W & \\
t_{klp}:&\mathcal{C}_{p-1}(\P(E_{k})\times_{W}\P(E_{l})) & \rightarrow & \mathcal{C}_{p}(\P(E_{k}\oplus E_{l})) & (k<l,\ k,l=1,2,3) \\
t_{p}:&\mathcal{C}_{p-2}(\P(E_{1})\times_{W}\P(E_{2})\times_{W}\P(E_{3})) & \rightarrow & \mathcal{C}_{p}(\P(E_{1}\oplus E_{2}\oplus E_{3})) &\\
\end{array}
$$
where $p$'s and  $q$'s are projections of the respective bundles, and the 
 $i$'s are the canonical inclusions.  The following diagram summarizes the situation:
 
\vspace{.5cm}

\vspace{.5cm}
\centerline{
\xymatrix{
 & \P(E_{2}) \ar@{^(->}[lddd]^(.3){i_{12}^{2}}|!{[ld];[dd]}\hole \ar@{^(->}[rrdd]^(.3){i_{23}^{2}}|!{[rr];[dd]}\hole & & \P(E_{3}) \ar@{^(->}[dd]^{i_{23}^{3}} \ar@{^(->}[lldd]^(.6){i_{13}^{3}} \\
\P(E_{1}) \ar@{^(->}[dd]^{i_{12}^{1}} \ar@{^(->}[dr]^(.7){i_{13}^{1}} & & & \\
 & \P(E_{1} \oplus E_{3}) \ar@{^(->}[rd]^{i_{13}} \ar[rdd]^(.8){p_{13}}|!{[ld];[rd]}\hole & & \P(E_{2} \oplus E_{3}) \ar@{^(->}[ld]^{i_{23}} \ar[ldd]^{p_{23}} \\
\P(E_{1} \oplus E_{2}) \ar[rrd]^{p_{12}} \ar@{^(->}[rr]^{i_{12}} & & \P(E_{1} \oplus E_{2} \oplus E_{3}) \ar[d]^q & \\
 & & W & 
}
}

\vspace{.5 cm}

 Then we have the following morphisms of monoids with finite multiplication:

$$\begin{array}{rrclc}
i_{kp}:&\Pi_{p}(\P(E_{k})) & \rightarrow & \Pi_{p}(\P(E_{1}\oplus E_{2}\oplus E_{3})) & (k=1,2,3) \\
& & & (\mbox{induced by } i_{k})\\
i_{klp}:&\Pi_{p}(\P(E_{k}\oplus E_{l})) & \rightarrow & \Pi_{p}(\P(E_{1}\oplus E_{2}\oplus E_{3})) & (k<l,\ k,l=1,2,3) \\
& & & (\mbox{induced by } i_{kl}) \\
\varphi_{klp}:&\Pi_{p-1}(\P(E_{k})\times_{W}\P(E_{l})) & \rightarrow & \Pi_{p}(\P(E_{k}\oplus E_{l})) & (k<l,\ k,l=1,2,3) \\
& & & (\mbox{induced by } t_{kl}) \\
\varphi_{p}:&\Pi_{p-2}(\P(E_{1})\times_{W}\P(E_{2})\times_{W}\P(E_{3})) & \rightarrow & \Pi_{p}(\P(E_{1}\oplus E_{2}\oplus E_{3})) &\\
& & & (\mbox{induced by } t_{p}) \\
\end{array}
$$
Then we have the following induced morphism:
\begin{displaymath}
\label{fip}	
\Psi_{p}:\left\{
\begin{array}{l}
 \Pi_{p-2}(\P(E_{1})\times_{W}\P(E_{2})\times_{W}\P(E_{3})) \times \Pi_{p-1}(\P(E_{1})\times_{W}\P(E_{2})) \\
 \times \Pi_{p-1}(\P(E_{1})\times_{W}\P(E_{3})) \times \Pi_{p-1}(\P(E_{2})\times_{W}\P(E_{3})) \\
 \times \Pi_{p}(\P(E_{1})) \times \Pi_{p}(\P(E_{2})) \times \Pi_{p}(\P(E_{3})) 
\end{array}
\right\} \rightarrow \Pi_{p}(\P(E_{1}\oplus E_{2}\oplus E_{3}))
\end{displaymath}
Where 
$(a,b,c,d,e,f,g) \longrightarrow \Psi(a,b,c,d,e,f,g)=\varphi_{p}(a)+i_{12p}\circ \varphi_{12p}(b)+i_{13p}\circ \varphi_{13p}(c)+i_{23p}\circ \varphi_{23p}(d)+i_{1p}(e)+i_{2p}(f)+i_{3p}(g).$
Now we proceed to define the  maps  $t_{klp} \mbox{ and } t_{p}$. 
Denote by $L_{1}, \ L_{2} \mbox{ and } L_{3}$ the lineal tautological bundles  $\calo_{E_{1}}(-1), \ \calo_{E_{2}}(-1)$ and $\calo_{E_{3}}(-1)$  
over $\P(E_{1}), \ \P(E_{2}) \mbox{ and } \P(E_{3})$, 
respectively and let $\pi_{1}, \ \pi_{2} \mbox{ and } \pi_{3}$ be the respective projections  from $$\P(E_{1}) \times_{W} \P(E_{2}) \times_{W} \P(E_{3})$$ to $$\P(E_{1}), \ \P(E_{2}) \mbox{ and } \P(E_{3})$$
 The  $\P^{2}$-bundle 
$\pi:\P(\pi_{1}^{*}(L_{1}) \oplus \pi_{2}^{*}(L_{2}) \oplus \pi_{3}^{*}(L_{3})) \rightarrow \P(E_{1}) \times_{W} \P(E_{2}) \times_{W} \P(E_{3})$ 
is the blow-up of  $\P(E_{1} \oplus E_{2} \oplus E_{3})$ on $\P(E_{1} \oplus E_{2}) \coprod \P(E_{1} \oplus E_{3}) \coprod \P(E_{2} \oplus E_{3})$, 
which we denote by $Q$ . Let $b:Q \rightarrow \P(E_{1} \oplus E_{2} \oplus E_{3})$  be the blow-up map. \\
Since $\pi$ is a flat map of relative dimension $2$, and $b$ 
is a proper map, we have two algebraic continuous homomorphisms given by the flat pull-back 
$$\pi^{*}:\mathcal{C}_{p-2}(\P(E_{1}) \times_{W} \P(E_{2}) \times_{W} \P(E_{3})) \rightarrow \mathcal{C}_{p}(Q)$$
 and the proper push-forward
$$b_{*}:\mathcal{C}_{p}(Q) \rightarrow \mathcal{C}_{p}(\P(E_{1} \oplus E_{2} \oplus E_{3})).$$

In the same way, for $k<l,\ k,l=1,2,3$, we denote by  $\pi_{k}^{kl} \mbox{ and } \pi_{l}^{kl}$ the respectively projections of
 $\P(E_{k}) \times_{W} \P(E_{l})$ on $\P(E_{k}) \mbox{ and } \P(E_{l})$. The $\P^{1}$-bundle $\pi_{kl}:\P(\pi_{k}^{kl*}(L_{k}) \oplus \pi_{l}^{kl*}(L_{l})) \rightarrow \P(E_{k} \times_{W} \P(E_{l}))$ 
 is precisely the blow-up of $\P(E_{k} \oplus E_{l})$ on $\P(E_{k}) \coprod \P(E_{l})$, 
 which we denote by  $Q_{kl}$. Let $b_{kl}:Q_{kl} \rightarrow \P(E_{k} \oplus E_{l})$ be the blow-up map. \\
Since $\pi_{kl}$   is a flat map of relative dimension  $1$, and $b_{kl}$ 
is a proper map, we have two continuous homomorphisms given by the flat pull-back
$$\pi_{kl}^{*}:\mathcal{C}_{p-1}(\P(E_{k}) \times_{W} \P(E_{l})) \rightarrow \mathcal{C}_{p}(Q_{kl})$$
and the proper push-forward 
$$b_{kl*}:\mathcal{C}_{p}(Q_{kl}) \rightarrow \mathcal{C}_{p}(\P(E_{k} \oplus E_{l})).$$

\begin{definition}
 The maps $t_{p} \mbox{ and } t_{klp}$ are defined as the  compositions  
 $t_{p}=b_{*} \circ \pi^{*} \mbox{ and } t_{klp}=b_{kl*} \circ \pi_{kl}^{*}$.
\end{definition}

The rest of the section is dedicated to prove Theorem~\ref{product:thm}. 

\begin{proof}
 Consider the action of
 $(\c^{*})^{2}$ on $\P(E_{1} \oplus E_{2} \oplus E_{3})$ 
 given by taking the product of each scalar with two of the factors of   $E_{1} \oplus E_{2} \oplus E_{3}$. 
 The set of fixed points of $\P(E_{1} \oplus E_{2} \oplus E_{3})^{(\c^{*})^{2}}$  is  $\P(E_{1})\coprod \P(E_{2}) \coprod \P(E_{3})$.
This action induces a continuous algebraic  action on $\mathcal{C}_{p}(\P(E_{1} \oplus E_{2} \oplus E_{3}))$; 
 and our next step is to identify its fixed points set. Consider the map

$$\psi_{p}:\left\{
\begin{array}{l}
 \C_{p-2}(\P(E_{1})\times_{W}\P(E_{2})\times_{W}\P(E_{3})) \times \C_{p-1}(\P(E_{1})\times_{W}\P(E_{2})) \\
 \times \C_{p-1}(\P(E_{1})\times_{W}\P(E_{3})) \times \C_{p-1}(\P(E_{2})\times_{W}\P(E_{3})) \\
 \times \C_{p}(\P(E_{1})) \times \C_{p}(\P(E_{2})) \times \C_{p}(\P(E_{3})) 
\end{array}
\right\} \rightarrow \C_{p}(\P(E_{1}\oplus E_{2}\oplus E_{3}))$$
defined by 
$\psi(a,b,c,d,e,f,g)=t_{p}(a)+i_{12p}\circ t_{12p}(b)+i_{13p}\circ t_{13p}(c)+i_{23p}\circ t_{23p}(d)+i_{1*}(e)+i_{2*}(f)+i_{3*}(g).$

We are going to prove that  $\psi_{p}$ is a homeomorphism in the set of fixed points 
 $(\mathcal{C}_{p}(\P(E_{1} \oplus E_{2} \oplus E_{3})))^{(\c^{*})^{2}}$.
If an element $\s=\sum_{i}{n_{i}V_{i}}\in \mathcal{C}_{p}(\P(E_{1} \oplus E_{2} \oplus E_{3})$ 
is fixed by the action, then each of its irreducible components must be invariant under the action. Furthermore, an irreducible invariant variety 
$V\subset \P(E_{1} \oplus E_{2} \oplus E_{3})$ can be of three types:  
\begin{enumerate}
\item Those whose general points are fixed under the action and therefore all the points of $V$ are fixed.
\item Those whose general points have non-trivial orbits of dimension 1. 
\item Those whose general points have non-trivial orbits of dimension 2.
\end{enumerate}

The Chow monoids $\mathcal{C}_{p}(X)$ of a variety $X$ are freely generated from the irreducible subvarieties of dimension $p$   
of $X$. Given cycles $\s_{k}\in \mathcal{C}_{p}(\P(E_{k})), \ k=1,2,3$, the support of $i_{k*}$ is contained in $\P(E_{k})$, it follows that  
the images of $\mathcal{C}_{p}(E_{1}), \mathcal{C}_{p}(E_{2}) \mbox{ and } \mathcal{C}_{p}(E_{3})$ under $i_{1}, i_{2} \mbox{ and } i_{3}$ are 
freely generated by the disjoint subset of the set of generators of 
$\mathcal{C}_{p}(\P(E_{1} \oplus E_{2} \oplus E_{3})$, which are varieties of the first type.

On the other hand, since the support of  $i_{kl*}$ is contained in $\P(E_{k} \oplus E_{l}), \ k,l=1,2,3$ with $ k \leq l$, then the image of $\P(E_{k} \oplus E_{l})$ is freely generated by disjoint subset of the set of generators
 $\mathcal{C}_{p}(\P(E_{1} \oplus E_{2} \oplus E_{3})$, and since the restriction of the action to $\P(E_{k} \oplus E_{l})$ 
 coincides with the  action of $\c^{*}$, we have that the elements of these set are of type 1 and 2, 
 (\cite{eli-lim}). 
 
 Since the irreducible varieties of type 1 that are contained in the image of  $i_{kl*}$ are some of the ones mentioned in
 the last paragraph,  it only remains to see how are the ones of type 2. 
 But these are of the form $i_{kl*} \circ t_{klp}(Z)$ for some $(p-1)$-dimensional variety  of
 $\P(E_{k}) \times_{W} \P(E_{l})$ by the proof of Theorem 5.1 in (\cite{eli-lim}). 

Now, given a ($p-2$)-dimensional subvariety $Z$ of $\P(E_{1}) \times_{W} \P(E_{2}) \times_{W} \P(E_{3})$ we have that the inverse image 
$\pi^{-1}(Z)$ is a $p$-dimensional subvariety of $Q$, whose points out of the exceptional divisor of the blow-up map $b$ 
have irreducible orbits of dimension 2. Since $b$ is a $(\c^{*})^{2}$-equivariant birational map, we have that the image 
$b(\pi^{-1}(Z))$ is an irreducible subvariety of  $\P(E_{1} \oplus E_{2} \oplus E_{3})$ of type 3. 

From this we obtain that the images of $i_{i*}, i_{2*}, i_{3*}, i_{12*} \circ t_{12}, i_{13*} \circ t_{13p}, i_{23*} \circ t_{23p}, t_{p}$
are generated freely by the disjoint subsets of the generators of $(\mathcal{C}_{p}(\P(E_{1} \oplus E_{2} \oplus E_{3}))^{(\c^{*})^{2}}$, 
therefore  $\psi_{p}$ is injective. 

In order to prove the surjectivity we only need to show that all invariant irreducible subvarieties $V\subset \P(E_{1} \oplus E_{2} \oplus E_{3})$
of type 3 are of the form $b(\pi^{-1}(Z))$ for some subvariety $(p-2)$-dimensional of 
$\P(E_{1}) \times_{W} \P(E_{2}) \times_{W} \P(E_{3})$; The remaining 
cases are consequence of the proof of theorem 5.1 in 
 \cite{eli-lim}.

Let $\tilde{V}\subset Q$ be the proper transform of $V$ under $b$, and define 
$Z:=\pi(\tilde{V})\subset \P(E_{1}) \times_{W} \P(E_{2}) \times_{W} \P(E_{3})$.
The general points of $V$ gave orbits of dimension 2, then the general points of $\tilde{V}$ also have dimension 2, this shows that the general fibre of  $\pi\rvert_{\tilde{V}}:\tilde{V} \rightarrow Z$ has dimension 2. In particular we have that 
$\tilde{V}=\pi^{-1}(Z)$ and since $\pi \left. \right|_{\tilde{V}}$ is a projective bundle, and $b$ is a birational maps from 
$\tilde{V}$ to $V$, we conclude that $t_{p}(Z)=b_{*} \circ \pi^{*} (Z)=V$.

The last argument shows that $\psi_{p}$ is an algebraic continuous  bijection from its domain to the set of fixed points
 $(\mathcal{C}_{p}(\P(E_{1} \oplus E_{2} \oplus E_{3}))^{(\c^{*})^{2}}.$ Observe that for each $\a \in \prod_{p}{(\P(E_{1} \oplus E_{2} \oplus E_{3}))}$
 the set of fixed points  $(\mathcal{C}_{p,\a}(\P(E_{1} \oplus E_{2} \oplus E_{3}))^{(\c^{*})^{2}}$
 is an algebraic subset of $(\mathcal{C}_{p,\a}(\P(E_{1} \oplus E_{2} \oplus E_{3}))$, not necessarily connected, and thus 
 we can also write it as a countable disjoint union of projective varieties. We use the following lemma 
 to see that $\psi_{p}$ is a homeomorphism onto its image. 
 
 \begin{lemma}
\label{lem:tech}
\cite[Lemma 5.4]{eli-lim}
Let $X = \coprod_{i\in \N} X_i$ and $Y=\coprod_{j\in \N} Y_j $ be spaces
which are a countable disjoint union of connected projective varieties,
and let $f: X \rightarrow Y$ be a continuous map such that the
restriction  $f\lvert_{{X_i}}$ is an algebraic continuous map from $X_i$
into some $Y_j$. If $f$ is a bijection, then it is a homeomorphism in
the classical topology.
\end{lemma}

 It is a general known result that the Euler characteristic of a variety with an algebraic torus action is equal to the Euler characteristic 
 of the fixed points set (for example see Lawson and Yau \cite{law&yau-hosy}). Therefore, given 
 $\a \in \prod_{p}{((\P(E_{1} \oplus E_{2} \oplus E_{3}))}$, it follows that 
$$\chi\left( \mathcal{C}_{p,\a}(\P(E_{1} \oplus E_{2} \oplus E_{3})) \right) = \chi \left( \mathcal{C}_{p,\a}(\P(E_{1} \oplus E_{2} \oplus E_{3}))^{(\c^{*})^{2}} \right).$$ 
On the other hand, if $\Psi_{p}$ is the induced morphism by $\psi_{p}$ between the connected components of the monoids, then  

\begin{setlength}{\multlinegap}{0pt}
\begin{multline*}
\mathcal{C}_{p,\a}(\P(E_{1} \oplus E_{2} \oplus E_{3}))^{(\c^{*})^{2}}= \\
\prod_{(a,b,c,d,e,f,g)\in \Psi_{p}^{-1}(\a)} \left(
\begin{array}{l}
 \mathcal{C}_{p-2,a}(\P(E_{1})\times_{W}\P(E_{2})\times_{W}\P(E_{3})) \times \mathcal{C}_{p-1,b}(\P(E_{1})\times_{W}\P(E_{2})) \\
 \times \mathcal{C}_{p-1,c}(\P(E_{1})\times_{W}\P(E_{3})) \times \mathcal{C}_{p-1,d}(\P(E_{2})\times_{W}\P(E_{3})) \\
 \times \mathcal{C}_{p,e}(\P(E_{1})) \times \mathcal{C}_{p,f}(\P(E_{2})) \times \mathcal{C}_{p,g}(\P(E_{3})) 
\end{array} \right)
\end{multline*}
\end{setlength}

\noindent since  $\psi_{p}$ is a homeomorphism onto its image. 
Taking the Euler characteristic on both sides the theorem is proven. 
\end{proof}

The following generalization  can be proved in similar way:

\begin{theorem}
\label{products:thm}
Let $E_{i}, \ i=1,\ldots,n$, vector bundles over a complex projective variety $W$ of ranks 
$e_{i}$, respectively, and we assume that  $n-1 \leq p\leq e_{1}+e_{2}+e_{3}-1$. 
Then, the $p$-th Euler-Chow function of $\P(\bigoplus_{i=1}^{n} E_{i})$ is given by 
\vspace{.5cm}

 \begin{setlength}{\multlinegap}{0pt}
\begin{multline*}
E_{p}(\P(\bigoplus_{i=1}^{n} E_{i}))  =  \\
\Psi_{p\#}\left(
E_{p-n+1}(\times_{i=1}^{n} \P(E_{i})) \odot \right. 
  \left. \left(\odot_{i=1}^{n} E_{p-n+2}(\times_{j=1,\hat{i}}^{n} \P(E_{j}))\right)\odot \cdots \odot \left(\odot_{i=1}^{n} E_{p}(\P(E_{i}))\right)
\right).
\end{multline*}
\end{setlength}
\end{theorem}

\section{The Euler-Chow series for ruled surfaces}
\label{ruled-surfaces}

This section is dedicated to prove the following

\begin{theorem}
\label{thm:ruled}
Let $X$ be a ruled surface over curve $C$, then its Euler-Chow series are given by
\begin{align*}
E_{0}(X)&=\left(\frac{1}{1-t}\right)^{\chi(X)}=\left(\frac{1}{1-t}\right)^{4-4g},\\
E_{1}(X)&=\left( \frac{1}{1-t_{0}} \right)^{2-2g} \left( \frac{1}{1-t_{1}} \right) \left( \frac{1}{1-t_{0}^{-e}t_{1}} \right),\\
E_{2}(X)&=\frac{1}{1-t}.
\end{align*}
\end{theorem}

It is simple to compute $E_0(X)$ and $E_2(X)$. From \cite{mcd-sym} we know that 
$$
E_0(X) \, = \, \left(\frac{1}{1-t}\right)^{\chi(X)} \, = \, \left( \frac{1}{1-t} \right)^{4-4g}
$$
and since a ruled surface is connected of dimension $2$, we have
$$
E_2(X) \, = \, \frac{1}{1-t}.
$$

For dimension one, we first prove it for the case of decomposable ruled surfaces, then we proceed to the indecomposable case reducing it to the decomposable situation.

\subsection{Decomposable ruled surfaces}
\begin{proof}

Let $\p:X \cong \P(\cale) \rightarrow C$ be a ruled surface such that $\cale$ decompose as the direct sum of invertible sheaves over $C$. 
By \cite[V, 2.12 (a)]{har} we have that $\cale \cong \calo_{C} \oplus \call$ 
for some $\call \mbox{ with } \deg \call \leq 0$. Furthermore $e=-\deg \call \geq 0$.
Since $\calo_{C}$ is the trivial line bundle over  $C$ we proceed to apply theorem  \ref{product:thm} for two vector bundles. We have 
$$\P(\calo_{C}) \cong \P(\call) \cong \P(\calo_{C})\times_{C} \P(\call) \cong C.$$
Using  the following
\begin{corollary}
\label{cor:psi}
\cite[Corollary 5.7]{eli-lim}
Let $E$ be a line bundle over $W$ which is generated by its global sections. Then the homomorphism
$$\Psi_p : \ax{p-1}{W}\oplus \ax{p}{W}\oplus \ax{p}{W} \rightarrow
\ax{p}{\pb{\eone}} \cong \ax{p-1}{W} \oplus \ax{p}{W}$$
sends $(\alpha, \beta,\gamma)$ to $(\alpha+\xi\cap \gamma,\ \beta + \gamma)$. \footnote{In the original paper there is a typo error in this last expression but in their proof is correct  }
\end{corollary}
\noindent We obtain that
$$\Psi_{1}:\Pi_{0}(C) \times \Pi_{1}(C) \times \Pi_{1}(C) \rightarrow \Pi_{1}(\P(\call \oplus \calo_{C})) \cong \Pi_{0}(C) \times \Pi_{1}(C)$$
is given by $\Psi(\a,\beta,\g)=(\xi \cap \g + \a,\beta +\g)$, 
where $\xi=c_{1}\left(\calo_{\P(\calo_{C}\oplus \call)}(1)\right)=c_{1}(\call(C_{0}))=C_{0}$. \\

Since $\g \in \Pi_{1}(C)$ and $C$ is a variety of dimension one, we have  $\g=c[C]=c[C_{0}]$. Therefore $\xi \cap \g=C_{0} . nC_{0}=-ne$. \\

From here we obtain $\Psi_{1}(a,b[C],c[C])=(a-ce,(b+c)[C])$, therefore we can identify  $\Psi_{1}$ by $(a,b,c)\mapsto (a-ce,b+c)$. 

We have $H_{2}(\P(\cale))\cong \z^{2}$, and for each generator we associate variables $t_{0}, t_{1}$. We identify
$(\a_{0},\a_{1}) \in \Pi_{1}(\P(\cale) \mbox{ with } t_{0}^{\a_{0}}t_{1}^{\a_{1}}$. We obtain \\

\begin{align*}
E_{1}(\P(\call \oplus \calo_{C})) & = \displaystyle \sum_{\a_{0},\a_{1}}E_{1}(\P(\call \oplus \calo_{C}))(\a_{0},\a_{1})t_{0}^{\a_{0}}t_{1}^{\a_{1}} \\
 & = \displaystyle \sum_{\a_{0},\a_{1}}\Psi_{1\#}(E_{0}(\P(\call)) \odot E_{1}(\P(\call)) \odot E_{1}(C))(\a_{0},\a_{1})t_{0}^{\a_{0}}t_{1}^{\a_{1}} \\
 & = \displaystyle \sum_{\a_{0},\a_{1}}\left( \sum_{(a,b,c)\in \Psi_{1}^{-1}}(E_{0}(C)(a) \cdot E_{1}(C)(b) \cdot E_{1}(C)(c)) \right)t_{0}^{\a_{0}}t_{1}^{\a_{1}} \\
 & = \displaystyle \sum_{\a_{0},\a_{1}}\left( \sum_{\substack{a-ce=\a_{0} \\ b+c=\a_{1}}}(E_{0}(C)(a) \cdot E_{1}(C)(b) \cdot E_{1}(C)(c)) \right)t_{0}^{a-ce}t_{1}^{b+c} \\
 & = \displaystyle \left(\sum_{a\geq 0}{E_{0}(C)(a)\cdot t_{0}^{a}}\right) \left(\sum_{b\geq 0}{E_{1}(C)(b)\cdot t_{1}^{b}}\right) \left(\sum_{c\geq 0}{E_{1}(C)(c)\cdot (t_{0}^{-e}t_{1})^{c}}\right)
\end{align*}
and thus
$$E_{1}(\P(\call \oplus \calo_{C}))=\left( \frac{1}{1-t_{0}} \right)^{2-2g} \left( \frac{1}{1-t_{1}} \right) \left( \frac{1}{1-t_{0}^{-e}t_{1}} \right)$$
\end{proof}

\subsection{Indecomposable ruled surfaces}

Before going through the proof, we will need here, and in the next section, the following

\begin{definition}
The group of all algebraic $p$-cycles on $X$ modulo
algebraic equivalence is denoted $\alg{p}{X}$, and the submonoid of
$\alg{p}{X}$ generated by the classes of cycles with non-negative
coefficients is denoted by $\algp{p}{X}$; cf. Fulton \cite[\S 12]{ful-inter}.
We use the notation $a\aeq b$ to express that two cycles $a,b$ are
algebraically equivalent.
\end{definition}

\begin{proof}

As $X$ is a ruled surface over $C$, then it is birationally equivalent to $C \times \P^{1}$ (\cite[V, 2.2.1]{har}). We can define an algebraic continuous action from $\mathbb{C}^{*}$ on $\P^{1}$ multiplying the first homogeneous coordinate. This action induces an algebraic continuous action of $\mathbb{C}^{*}$ on $C \times \P^{1}$ and therefore induces one on $X$ (\cite[II, 7.15]{har}). Under this action the fibers are stable.

The fixed point set of $C \times \P^{1}$ under the action is a union of two disjoint sections, and hence the fixed point set of $X$ under the action is this union of two sections, which aren't disjoint, plus a finite set of points:
$$X^{\c^{*}}=C_{0} \cup C_{1} \cup \{p_{1},\ldots,p_{r}\} \mbox{ where\, } r=|C_{0} \cap C_{1}|.$$

This action induces an algebraic continuos action on $\calc_{1}(X)$, and we are interested in identifying the set of fixed points in this case. 
Consider the map
$$\phi_{1}:\calc_{0}(C) \times \calc_{1}(C) \times \calc_{1}(C) \rightarrow \calc_{1}(X)$$
defined by $\phi_{1}(a,b,c)=\pi^{*}(a)+\sigma_{0*}b+\sigma_{1*}c$, where
 $\sigma_{0} \mbox{ and } \sigma_{1}$ are respectively the sections that correspond to $C_{0} \mbox{ and } C_{1}$ .

We want to prove that $\phi_{1}$ is an homeomorphism on the set of fixed points $\calc_{1}(X)^{\c^{*}}$.
 
The invariant irreducible subvarieties of dimension $1$ are $C_{0}, \ C_{1},$ in fact, all their points are fixed under the action, and they are also the fibers  of 
$\pi$ (all their general points are in the same orbit under the action). 
From this, it is clear that  $\phi_{1}$ is injective and surjective over  $\calc_{1}(X)$.

Note that for each $\a \in \Pi_1(X)$, the set of fixed points  $\calc_{1, \a}(X)^{\c^{*}}$ is an algebraic 
 subset of $\calc_{1, \a}(X)$ (not necessarily connected). Hence it can be written as the disjoint union of projective varieties.
  
 Therefore, by lemma \ref{lem:tech}
 we conclude that $\phi_{1}$ is an homeomorphism on its image. 
 
 We know by \cite{law&yau-hosy} that the Euler characteristic of an algebraic variety is equal to the Euler characteristic of its fixed point set. We have

$$\chi(\calc_{1,\a}(X))=\chi(\calc_{1,\a}(X)^{\c^{*}}).$$

\noindent On the other hand, let $\Psi_{1}$ be the morphism induced by $\phi_{1}$ between the monoids of the connected components. We have
$$\calc_{1,\a}(X)^{\c^{*}}=\coprod_{(a,b,c)\in \Psi_{1}^{-1}(\a)}\calc_{0,a}(C) \times \calc_{1,b}(C) \times \calc_{1,c}(C),$$

$$\chi(\calc_{1,\a}(X))=\sum_{(a,b,c)\in \Psi_{1}^{-1}(\a)}\chi(\calc_{0,a}(C)) \cdot \chi(\calc_{1,b}(C)) \cdot \chi(\calc_{1,c}(C)).$$
Therefore
$$E_{1}(X)=\Psi_{1\#}(E_{0}(C) \odot E_{1}(C) \odot E_{1}(C)).$$
Furthermore we have the following morphism
$$ 
\cala_1(X)  \stackrel{\simeq}{\longrightarrow} \cala_1(\P(\call(\e) \oplus \calo_{C})) \quad\text{given by}\quad
nC_{0}+\p^{*}b  \mapsto  nC'_{0}+\p'^{*}b.
$$
This morphism can be restricted to the effective part and to its pathwise connected components. Using the results obtained for decomposable ruled surfaces we obtain

$$E_{1}(X)=\left( \frac{1}{1-t_{0}} \right)^{2-2g} \left( \frac{1}{1-t_{1}} \right) \left( \frac{1}{1-t_{0}^{-e}t_{1}} \right).$$

\end{proof}

\section{The Euler-Chow series for a normal rational scroll}
\label{scrolls}

Let  $E_{1}, \ E_{3}$ be line bundles over a projective variety $W$, and let $E_{2}=\uno$ be the trivial line bundle.
We have
\begin{align*}
W & =  \P(E_{1})=\P(E_{2})=\P(E_{3}) \\
  & =  \P(E_{1})\times_{W} \P(E_{2})=\P(E_{1})\times_{W} \P(E_{3})=\P(E_{2})\times_{W} \P(E_{3}) \\
	& =  \P(E_{1})\times_{W} \P(E_{2})\times_{W} \P(E_{3}),
\end{align*}
and the inclusion
 $i_{k}:\P(E_{k})\rightarrow \P(E_{1}\oplus E_{2}\oplus E_{3}), \ k=1,2,3,$  are sections of the projective bundle
  $\P(E_{1}\oplus E_{2}\oplus E_{3})$ over $W$.

We write  $\xi=c_{1}(\calo_{\P(E_{1}\oplus E_{2}\oplus E_{3})}(1))$ and  $\xi '=c_{1}(\calo_{\P(E_{1}\oplus \uno)}(1))$.
We have the following isomorphism
\begin{equation}
 \label{iso chow}
  \begin{array}{rcl}
   T:\cala_{p-2}(W)\oplus \cala_{p-1}(W)\oplus \cala_{p}(W) & \rightarrow & \cala_{p}(\P(E_{1}\oplus E_{2}\oplus E_{3})) \\
	 (\a,\b,\g) & \mapsto & q^{*}\a + \xi \cap q^{*}\b + \xi^{2} \cap q^{*}\g \\
	 &  & =q^{*}\a + i_{12*}p_{12}^{*}\b + i_{1*}p_{1}^{*}\g,
	\end{array}
\end{equation}
since
\begin{align*}
\xi^{2}\cap q^{*}\g & =  \xi \cap i_{12*}p_{12}^{*}\g=i_{12*}(\xi ' \cap p_{12}^{*}\g) \\ 
 & =  i_{12}^{*}(i_{1*}^{'}p_{1}^{*}\g)=i_{1*}p_{1*}\g
\end{align*}
this isomorphim is equal to the composition 
$$\cala_{p-2}(W)\oplus \cala_{p-1}(W)\oplus \cala_{p}(W) \rightarrow \cala_{p-2}(W) \oplus \cala_{p}(\P(E_{1}\oplus E_{2})) \rightarrow \cala_{p}(\P(E_{1}\oplus E_{2}\oplus E_{3})).$$ 
which restricts to an injection
$$T^{\geq}:\cala_{p-2}^{\geq}(W)\oplus \cala_{p-1}^{\geq}(W)\oplus \cala_{p}^{\geq}(W) \rightarrow \cala_{p}^{\geq}(\P(E_{1}\oplus E_{2}\oplus E_{3})).$$  
\begin{lemma}
The injection  $T^{\geq}$ is an isomorphism. 
If $\Pi_{*}(W)$ are monoids with cancelation law for all  $*$, then so are  $\Pi_{*}(\P(E_{1}\oplus E_{2}\oplus E_{3}))$. Equivalently, if the natural surjective maps 
	$\Pi_{p}(W) \rightarrow \cala_{p}^{\geq}(\P(E_{1}\oplus E_{2}\oplus E_{3}))$ are isomorphisms for all $p$,
	 then the following surjective maps are also isomorphisms $\Pi_{p}(\P(E_{1}\oplus E_{2}\oplus E_{3})) \rightarrow \cala_{p}^{\geq}(\P(E_{1}\oplus E_{2}\oplus E_{3}))$.
\end{lemma}
\begin{proof}
 All effective  $p$-cycles $a'$ are algebraic and effectively equivalent to a cycles of the form 
$$q^{*}(a)+i_{12*}p_{12}^{*}(b)+i_{13*}p_{13}^{*}(c)+i_{23*}p_{23}^{*}(d)+i_{1*}(e)+i_{2*}(f)+i_{3*}(g)$$
with $a\in \mathcal{C}_{p-2}(W), \ b,c,d \in \mathcal{C}_{p-1}(W), \ e,f,g \in \mathcal{C}_{p}(W)$. 
We have
 
\begin{align*}
	 i_{2*}(f) & =  i_{12*}\circ i_{2*}^{'} \\
	 & \equiv  i_{12*}(i_{1*}^{'}(f)+p_{12}^{*}(Z^{'}\cap f)) \\
	 & =  i_{12*}\circ i_{1*}^{'}(f)+i_{12*}\circ p_{12}^{*}(Z^{'}\cap f) \\
	 & =  i_{1*}(f)+i_{12*}p_{12}^{*}(Z^{'}\cap f).
\end{align*}
In a similar form we obtain 
$$i_{3*}(g)\equiv i_{1*}(g)+i_{12*}p_{12}^{*}(Z^{'}\cap g).$$
Since $E_{1}$ is generated by its sections, we can find a section $s:W\rightarrow E_{1}$ whose zero locus  $Z\subset W$ intersects properly to $c$. Let $\tilde{s}:W\rightarrow \P(E_{1}\oplus \uno \oplus E_{3})$ be the composition $\iota \circ s$, where $\iota:E_{1}\rightarrow \P(E_{1}\oplus \uno \oplus E_{3})$ is the open inclusion. Then, the closure of the orbit of $\tilde{s}_{*}c$ contains two fixed points: $i_{12*} \circ p_{12}^{*}(c)+q^{*}(Z\cap c)$ and $i_{13*}\circ p_{13}^{*}(c)$. Similarly for $i_{23*}\circ p_{23}^{*}(d)$. Then
	$$a'\equiv_{alg^{+}}q^{*}(a+Z\cap c +Z\cap d)+i_{12*}\circ p_{12}^{*}(b+c+d+Z^{'}\cap f +Z^{'}\cap g) + i_{1*}(e+f+g).$$
This proves the surjectivity of $T^{\geq}$  and	
injectivity is proven in the same way as in \cite[Lemma 5.6]{eli-lim}.
\end{proof}

\begin{corollary}
\label{scrolls:cor}
Under the same hypothesis, the homomorphism
$$\begin{array}{rcl}
 \Psi_{p}:\left\{
\begin{array}{r}
 \Pi_{p-2}(\P(E_{1})\times_{W}\P(E_{2})\times_{W}\P(E_{3})) \times \\
 \Pi_{p-1}(\P(E_{1})\times_{W}\P(E_{2}))\times \Pi_{p-1}(\P(E_{1})\times_{W}\P(E_{3})) \times \\
 \Pi_{p-1}(\P(E_{2})\times_{W}\P(E_{3}))\times \\
 \Pi_{p}(\P(E_{1})) \times \Pi_{p}(\P(E_{2})) \times \Pi_{p}(\P(E_{3}))
\end{array}\right\}
 & \rightarrow & \Pi_{p}(\P(E_{1} \oplus \uno \oplus E_{3})) \\
 & & \simeq \Pi_{p-2}(W) \oplus \Pi_{p-1}(W) \oplus \Pi_{p}(W)
\end{array}$$
sends $(a,b,c,d,e,f,g)$ to $(a+\xi \cap c +\xi \cap d, \ b+c+d+\xi^{'}\cap f + \xi^{'}\cap g, \ e+f+g)$.

\end{corollary}

\begin{example}
 For
$W=\P^{1}, \ E_{1}=\calo_{\P^{1}}(h) \mbox{ with } h\geq0, \ E_{2}=E_{3}=\uno$
we have
 \begin{setlength}{\multlinegap}{0pt}
\begin{multline*}
 \Psi_{p}(a[\P^{p-2}],b[\P^{p-1}],c[\P^{p-1}],d[\P^{p-1}],e[\P^{p}],f[\P^{p}],g[\P^{p}]) \\
=  ((a+h \cdot c+h \cdot d)[\P^{p-2}],(b+c+d+h \cdot f + h \cdot g)[\P^{p-1}],(e+f+g)[\P^{p}]).
\end{multline*}
\end{setlength}
Therefore
\begin{setlength}{\multlinegap}{0pt}
\begin{multline*}
E_{p}(\P(\calo_{\P^{1}}(h) \oplus \uno \oplus \uno))= \\
\left(\frac{1}{1-t_{0}}\right)^{{n+1}\choose{p-1}} \left(\frac{1}{1-t_{1}}\right)^{{n+1}\choose{p}} \left(\frac{1}{1-t_{2}}\right)^{{n+1}\choose{p+1}} \left(\frac{1}{1-t_{0}^{-h}t_{1}}\right)^{2{{n+1}\choose{p}}} \left(\frac{1}{1-t_{1}^{-h}t_{2}}\right)^{2{{n+1}\choose{p+1}}}.
\end{multline*}
\end{setlength}
\end{example}

\providecommand{\bysame}{\leavevmode\hbox to3em{\hrulefill}\thinspace}


\end{document}